\documentclass[12pt,reqno]{amsart}

\usepackage{color}
\usepackage{hyperref}
\usepackage{url}
\usepackage{amsmath}
\usepackage{amssymb}
\usepackage{amstext}
\usepackage{comment}
\usepackage{stmaryrd}
\usepackage{latexsym}
\usepackage{enumitem}
\usepackage[left=3.5cm, right=3.5cm, headheight=14.0pt]{geometry}

\hypersetup{
    pdftitle={Cauchy-Davenport type inequalities, I},
    pdfauthor={Salvatore Tringali},
    pdfmenubar=false,
    pdffitwindow=true,
    pdfstartview=FitH,
    colorlinks=true,
    linkcolor=blue,
    citecolor=green,
    urlcolor=cyan
    pdfborder={0 0 0},
}

\usepackage{fancyhdr}
\newtheorem{theorem}{Theorem}
\newtheorem{lemma}[theorem]{Lemma}
\newtheorem{proposition}[theorem]{Proposition}
\newtheorem{corollary}[theorem]{Corollary}
\newtheorem{conjecture}[theorem]{Conjecture}

\theoremstyle{definition}

\newtheorem*{claim*}{Claim}

\renewcommand{\AA}{\mathbb A}

\providecommand{\ord}{\mathop{\rm ord}\nolimits}

\providecommand{\NNb}{\mathbf{N}}

\providecommand{\ZZb}{\mathbf{Z}}

\renewcommand{\emptyset}{\varnothing}

\newcommand{\fixed}[2][1]{%
  \begingroup
  \spaceskip=#1\fontdimen2\font minus \fontdimen4\font
  \xspaceskip=0pt\relax
  #2%
  \endgroup
}

\begin{document}
\title{Cauchy-Davenport type inequalities, I}
\author{Salvatore Tringali}
\address{Institute for Mathematics and Scientific Computing, University of Graz | Heinrichstr. 36, 8010 Graz, Austria}
\email{salvatore.tringali@uni-graz.at}
\urladdr{http://143.50.47.129/tringali/home.html}
%
% 05E15: Combinatorial aspects of groups and algebras
% 11B13: Additive bases, including sumsets
% 20D60: Arithmetic and combinatorial problems
% 20E99: None of the above, but in this section (Group theory and generalizations)
%
\begin{abstract}
Let $\mathbb G = (G, +)$ be a group (either abelian or not). Given $X, Y \subseteq G$, we denote by $\langle Y \rangle$ the subsemigroup of $\mathbb G$ generated by $Y$, and we set $$\gamma(Y) := \sup_{y_0 \in Y} \inf_{y_0 \ne y \in Y} {\rm ord}(y - y_0)$$ if $|Y| \ge 2$ and $\gamma(Y) := |Y|$ otherwise. We prove that if $\langle Y \rangle$ is com\-mu\-ta\-tive, $Y$ is non-empty, and $X+2Y \neq X + Y + y$ for some $y \in Y$, then
$$
|X+Y| \ge |X|+\min(\gamma(Y), |Y| - 1).
$$
Actually, this is obtained from a more general result, which improves on previous work of the author on sumsets in cancellative semigroups, and yields a comprehensive generalization, and
in some cases a considerable strengthening, of various additive theorems, notably including the Chowla-Pillai theorem (on sumsets in finite cyclic groups) and the specialization to abelian groups of the Hamidoune-Shatrowsky theorem.
\end{abstract}
\maketitle
\thispagestyle{empty}
\section{Introduction}
\label{sec:intro}
Let $\AA = (A, +)$ be, unless otherwise specified, an additively written semigroup, viz. an ordered pair consisting of a set and a binary associative operation on it; note that, in this paper, ``additive'' does not imply ``commutative''.
We address the reader to \cite[\S{} 1.1]{How} for basic aspects of semigroup theory.

If $X_1, \ldots, X_n \subseteq A$, we let $X_1 + \cdots + X_n$ denote, as usual, the sumset, relative to $\mathbb A$, of the $n$-tuple $(X_1, \ldots, X_n)$, namely
$$
X_1 + \cdots + X_n := \{x_1 + \cdots + x_n: x_1 \in X_1, \ldots, x_n \in X_n\};
$$
we replace $X_i$ with $x_i$ in this notation if $X_i = \{x_i\}$ for some $i$, provided it does not cause confusion, and we use $n X_1$ for $X_1 + \cdots + X_n$ if $X_1 = \cdots = X_n$.

We write $\AA^\times$ for the set of units (or invertible elements) of $\AA$, and for $X \subseteq A$ we set $X^\times := X \cap \mathbb A^\times$ when there is no danger of ambiguity.

In particular, $A^\times = A$ if and only if $\mathbb A$ is a group or $A$ is empty, and $A^\times \ne \emptyset$ if and only if $\AA$ is a monoid, i.e. there exists a (provably unique) element $0_\mathbb{A} \in A$, labeled as the identity of $\mathbb A$, such that $x + 0_\mathbb{A} = 0_\mathbb{A} + x = x$ for all $x \in A$.
%(in which case $A^\times$ is a subgroup of $\mathbb A$, as it follows, e.g., from Lemma \ref{lem:old_units_distribution} below).

For $X, Y \subseteq A$ we let $X - Y := \{z \in A: X \cap (z + Y) \ne \emptyset\}$ and ${-Y}+X := \{z \in A: X \cap (Y+z) \ne \emptyset\}$, which extends the notion of difference set from groups to semigroups, at the cost that $X-Y$ or ${-Y}+X$ may be empty even if $X$ and $Y$ are not. We use, respectively, $X - y$ and ${-y} + X$ in place of $X - Y$ and ${-Y} + X$ if $Y = \{y\}$ and no confusion can arise; note that $X - y = \{x + (-y): x \in X\}$ if $y \in A^\times$, and similarly for ${-y} + X$.

Given $X \subseteq A$, we denote by $\langle X \rangle$ the smallest subsemigroup of $\mathbb A$ containing $X$, and we set $\langle\fixed[-0.6]{ \text{ }}\langle X \rangle\fixed[-0.6]{ \text{ }}\rangle := \langle X \cup \{-x: x \in X^\times\} \rangle$ and $\ord(X) := |\langle X \rangle|$. If $X = \{x\}$ and there is no risk of misunderstanding, we just write $\ord(x)$ in place of $\ord(X)$, and we use $\langle x \rangle$ in an analogous way.

Lastly, we say that an element $z \in A$ is cancellable (in $\mathbb A$) if both the functions $A \to A: x \mapsto x + z$ and $A \to A: x \mapsto z + x$ are injective, and we refer to $\mathbb A$ as a cancellative semigroup if every $z \in A$ is cancellable.
\section{The Cauchy-Davenport constant of an \texorpdfstring{$n$-tuple}{n-tuple}}
\label{sec:CD-constants}
With the above in mind, we now introduce a quantity that happens to capture, as discussed below, interesting features of the ``combinatorial structure'' of $\mathbb A$.

To start with, we set, for every $X \subseteq A$, $\gamma(X) := |X|$ if $|X| \le 1$, otherwise
\begin{equation}
\label{equ:CD_constant}
\gamma(X) := \sup_{x_0 \in X^\times} \inf_{x_0 \ne x \in X} {\rm ord}(x - x_0),
\end{equation}
%
%where we adopt the convention that
where $\sup(\emptyset) := 0$ and $\inf(\emptyset) := \infty$;
note that \eqref{equ:CD_constant} can be slightly simplified if $\mathbb A$ is a group, by replacing $X^\times$ with $X$.

Throughout, we will see how to use $\gamma(X)$ to obtain, for $X, Y \subseteq A$, non-trivial lower bounds on $|X+Y|$. But first, for all $X_1, \ldots, X_n \subseteq A$ let
\begin{equation*}
\label{equ:def_of_large_CD}
\gamma(X_1, \ldots, X_n) :=
\left\{
\begin{array}{ll}
\!\! 0 & \text{if }X_i = \emptyset\text{ for some }i \\
\!\! \max_{1 \le i \le n\fixed[0.2]{ \text{ }}} \gamma(X_i) & \text{otherwise}
\end{array}
\right.\!\!\!;
\end{equation*}
we refer to $\gamma(X_1, \ldots, X_n)$ as the \textit{Cauchy-Davenport constant}, relative to $\mathbb A$, of the $n$-tuple $(X_1, \ldots, X_n)$. Occasionally, we may add a subscript `$\mathbb A$' to the right of the letter $\gamma$ in the above definitions if we need, for any reason, to be explicit about the semigroup on which they do actually depend.

The Cauchy-Davenport constant of a tuple was first introduced in \cite{Tring12}, though in a different notation and in a somewhat different form, and further investigated in \cite{Tring13}, as part of a broader program aimed at the extension of aspects of the theory on Cauchy-Davenport type inequalities from groups to more abstract settings, where certain properties (of groups) are no longer necessarily true.

In particular, the author proposed in \cite{Tring13} to prove (or disprove) the following:
\begin{conjecture}
\label{conj:1}
If $\AA$ is a cancellative semigroup and $X_1, \ldots, X_n \subseteq A$, then
\begin{equation*}
%\label{equ:conjecture}
|X_1+\cdots+X_n| \ge \min(\gamma(X_1, \ldots, X_n), |X_1| + \cdots + |X_n| + 1 - n).
\end{equation*}
\end{conjecture}
This is plainly true if $n = 1$ or $X_1^\times = \cdots = X_n^\times = \emptyset$, is straightforward when $|X_i| \ge \gamma(X_1, \ldots, X_n)$, and especially $|X_i| = \infty$, for some $i = 1, \ldots, n$ (see Lemma \ref{lem:trivialieties}\ref{item:trivialieties_(ii)} below), and has been so far confirmed in a couple more of cases:
\begin{enumerate}[label={\rm (\roman{*})}]
\item\label{item:partial_results} if $n = 2$ and each of $X_1$ and $X_2$ generates a commutative subsemigroup of $\mathbb A$, see \cite[Theorem 8 and Corollary 10]{Tring12};
\item if $\gamma(X_1) = \cdots = \gamma(X_n)$ and $\mathbb A$ is commutative (in addition to being cancellative), see the note added in proof at the end of \cite[\S{} 6]{Tring13}.
\end{enumerate}
The conjecture was first motivated by point \ref{item:partial_results} above and the following theorem, see \cite[Theorem 7]{Tring13}:
if $\mathbb A$ is a cancellative semigroup and $X, Y \subseteq A$, then
\begin{equation}
\label{equ:weaker_conjecture}
|X+Y| \ge \min(\gamma(X+Y), |X| + |Y| - 1).
\end{equation}
In fact, \eqref{equ:weaker_conjecture} is weaker than the bound provided by Conjecture \ref{conj:1} in the case of two summands,
as $\gamma(X, Y) \ge \gamma(X+Y)$ in general, and actually $\gamma(X, Y) \gg \gamma(X + Y)$ in many common situations, see \cite[Lemma 3 and Example 4]{Tring13}. (The
symbol `$\gg$' means that, for a suitable choice of the semigroup $\mathbb A$ and the sets $X$ and $Y$, the left-hand side can be made larger than the right-hand
side by an arbitrary factor.)

With all this said, here comes the main contribution of the present work, which is, in the first place, a strengthening of the Cauchy-Davenport theorem \cite{Cauchy1813, Daven35, Daven47}.
%(See below for further details.)
%
\begin{theorem}
\label{th:mainA062}
Assume $\mathbb A$ is a cancellative semigroup, and let $X, Y \subseteq A$ such that $\langle Y \rangle$ is com\-mu\-ta\-tive and $Y \ne \emptyset$. Then at least one of the following holds:
\begin{enumerate}[label={\rm (\roman{*})}]
\item\label{cor:main_condition_(i)} $|X + Y| \ge |X| + \min(\gamma(Y), |Y| - 1)$;
\item\label{cor:main_condition_(ii)} $X + 2Y = X + Y + \bar{y}$ for some $\bar{y} \in Y^\times$.
\end{enumerate}
\end{theorem}
Loosely speaking, the theorem says that, for all $X, Y \subseteq A$, and in the presence of cancellativity, $|X+Y|$ cannot be ``too small'', unless $X+Y$ has ``structure'', which
is made more precise by the next statement, whose proof we defer to the end of \S{} \ref{sec:preliminaries}.
\begin{proposition}
\label{prop:equivalent_form_with_generated_subgroup}
Assume $\mathbb A$ is a cancellative semigroup, and let $X, Y \subseteq A$ be finite sets such that $\langle Y \rangle$ is com\-mu\-ta\-tive and $Y^\times \ne \emptyset$. Then the following are equivalent:
\begin{enumerate}[label={\rm (\roman{*})}]
\item\label{cor:equ_condition_(i)} $X + 2Y = X + Y + \bar{y}$ for some $\bar{y} \in Y^\times$;
\item\label{cor:equ_condition_(ii)} $X + 2Y = X + Y + y$ for all $y \in Y$;
\item\label{cor:equ_condition_(iii)} $X + \langle\fixed[-0.6]{ \text{ }}\langle Y - \bar y \rangle\fixed[-0.6]{ \text{ }}\rangle = X + \langle Y - \bar y \rangle = X + Y - \bar y$ for every $\bar{y} \in Y^\times$.
\end{enumerate}
\end{proposition}
We use Theorem \ref{th:mainA062} and Proposition \ref{prop:equivalent_form_with_generated_subgroup} to prove a couple of corollaries: the first has essentially the same content of \cite[Theorem 8]{Tring12}, and deriving it from Theorem \ref{th:mainA062} shows, in the end, that the results of this paper subsume and strengthen \textit{all} those obtained in \cite{Tring12}; the second is reminiscent of an addition theorem of Y.~O.~Hamidoune \cite[p.~249]{Hami} we refer to as the Hamidoune-Shatrowsky theorem, as it is a generalization of an earlier (and weaker) result of L.~Shatrowsky \cite{Shat}.
\begin{corollary}
\label{cor:UDT_improved}
Assume $\mathbb A$ is a cancellative semigroup, and let $X, Y \subseteq A$ such that $X \ne \emptyset$ and $\langle Y \rangle$ is com\-mu\-ta\-tive. Then
$
|X + Y| \ge \min(\gamma(Y), |X| + |Y| - 1)$.
\end{corollary}
\begin{corollary}
\label{cor:HS}
Let $\mathbb A$ be a cancellative monoid with identity $0_\mathbb{A}$, and let $X, Y \subseteq A$ such that
$\langle Y \rangle$ is commutative and $X \cup (X + Y) \ne X + \langle\fixed[-0.6]{ \text{ }}\langle Y \rangle\fixed[-0.6]{ \text{ }}\rangle$. Then
\begin{equation}
\label{equ:stronger_than_HS}
|X \cup (X + Y)| \ge |X| + \min(\gamma(Y \cup \{0_\mathbb{A}\}), |Y| - \mathbf 1_Y(0_\mathbb{A})),
\end{equation}
where $\mathbf 1_Y(0_\mathbb{A}) := 1$ if $0_\mathbb{A} \in Y$, otherwise $\mathbf 1_Y(0_\mathbb{A}) := 0$.
\end{corollary}
As long as $\langle Y \rangle$ is commutative, Corollary \ref{cor:HS} is indeed stronger, and can be \textit{much} stronger, than the Hamidoune-Shatrowsky theorem, according to which we would rather have that if $\mathbb A$ is a group, $Y \ne \emptyset$, and $X \cup (X + Y) \ne X + \langle\fixed[-0.6]{ \text{ }}\langle Y \rangle\fixed[-0.6]{ \text{ }}\rangle$, then
% in the latter the right-hand side of \eqref{equ:abelian_shatrovskii} would be actually replaced by
\begin{equation}
\label{equ:hamidoune-shatrowsky_bound}
|X \cup (X + Y)| \ge |X| + \min(v(Y), |Y|),
\end{equation}
where $v(Y)$ is the minimal order of the elements of $Y$. In fact, there are two cases:
\begin{enumerate}[label={\rm (\roman{*})}]
\item $0_\mathbb{A} \in Y$. Then $X \subseteq X + Y$ and $v(Y) = 1$, hence \eqref{equ:hamidoune-shatrowsky_bound} simplifies to $|X + Y| \ge |X| + 1$. If $|Y| \ge 2$, this is (strictly) weaker than \eqref{equ:stronger_than_HS}, and actually \textit{much} weaker than \eqref{equ:stronger_than_HS} for (comparatively) large values of both $\gamma(Y \cup \{0_\mathbb{A}\})$ and $|Y|$ (which can be easily attained). Otherwise, $Y = \{0_\mathbb{A}\}$ and $X + \langle\fixed[-0.6]{ \text{ }}\langle Y \rangle\fixed[-0.6]{ \text{ }}\rangle = X + Y = X$, which, however, would be a contradiction.
\item $0_\mathbb{A} \notin Y$. Then $|Y| - \mathbf 1_Y(0_\mathbb{A}) = |Y|$, and of course $\gamma(Y \cup \{0_\mathbb{A}\}) \ge v(Y)$.
\end{enumerate}
On the other hand, the Hamidoune-Shatrowsky theorem holds, provided $\mathbb A$ is a group, without the additional assumption on $\langle Y \rangle$ made in Corollary \ref{cor:HS}, which leads us to believe that a more general version of Theorem \ref{th:mainA062} should be true.
%, though for the moment we do not have a precise conjecture in this regard.

Incidentally, let us mention here that, while the original proof of the Hamidoune-Shatrowsky theorem relies on Hamidoune's theory of atoms, our proof of Theorem \ref{th:mainA062}, and hence of Corollary \ref{cor:HS}, is essentially based on a non-commutative variant of the Davenport transform first considered, to our knowledge, in \cite[\S{} 4]{Tring12}.

Our last result is a special case of Theorem \ref{th:mainA062} and a strengthening of \cite[Corollary 15]{Tring12}, which is in turn a generalization of the Chowla-Pillai theorem (on sumsets in finite cyclic groups), see \cite[Theorem 1]{Chow} and \cite[Theorems 1--3]{Pil}.
%In fact, the Chowla-Pillai theorem is also generalized by the Hamidoune-Shatrowsky theorem, which, as we have already noted, is weaker, and in some cases much weaker, than Corollary \ref{cor:abelian_shatrovskii}, when it comes to abelian groups.
%
\begin{corollary}
\label{cor:pillai&chowla_improved}
Fix $n \in \mathbf N^+$. Denote by $(\ZZb_n, +)$ the additive group of the integers modulo $n$ and by $\pi$ the canonical projection $\ZZb \to \ZZb_n$. Let $X, Y \subseteq \ZZb_n$ be non-empty sets such that $X+2Y \ne X + Y + \bar y$ for some $\bar y \in Y$. Then
\begin{displaymath}
|X+Y| \ge |X| + \min(\delta_Y^{-1} n, |X|+|Y|-1),
\end{displaymath}
where $\delta_Y := 1$ if $|Y| = 1$, otherwise
$$
\delta_Y := \min_{y_0 \in \pi^{-1}(Y)} \fixed[0.05]{ \text{ }} \max_{y \in \pi^{-1}(Y),\fixed[0.2]{ \text{ }} \pi(y) \fixed[0.1]{ \text{ }} \ne \fixed[0.1]{ \text{ }} \pi(y_0)}  \gcd(n, y - y_0).
$$
\end{corollary}
We will prove Theorem \ref{th:mainA062} and its corollaries in \S{} \ref{sec:proofA062}, but first we need to gather together a few facts that play a role in the proofs: this is done in the next section.
\section{Preparations}
\label{sec:preliminaries}
%
%Some results and proofs are adapted from \cite[\S{} 4]{Tring13}, but significant differences occur in most cases.
We start with basic properties of semigroups that are readily adapted from the case of groups and used repeatedly in the sequel (with or without a comment).
\begin{lemma}
\label{lem:trivialieties}
The following hold:
\begin{enumerate}[label={\rm (\roman{*})}]
\item\label{item:trivialieties_(i)} If $z \in A$ is cancellable and $X \subseteq A$, then $|z+X|=|X+z|=|X|$.
\item\label{item:trivialieties_(ii)} If $X_1, \ldots, X_n \subseteq A$ and $X_i$ contains at least one cancellable element for each $i$, then
$\left|X_1 + \cdots + X_n\right| \ge \max_{1 \le i \le n} |X_i|$.
\item\label{item:trivialieties_(iii)} Let $z \in A$, and assume $z$ is cancellable and $n := \ord(z) < \infty$. Then $\mathbb A$ is a monoid and $nz$ is the identity of $\mathbb A$.
\item\label{item:trivialities_(iv)} Let $\mathbb A$ be a monoid and $X \subseteq A$, and let
$
z \in \mathcal C(X) \cap A^\times$, where
$$
\mathcal C(X) := \{z \in A: x + z = z + x \text{ for all }x \in X\}
$$
is the center of $X$ (in $\mathbb A$).
Then $-z \in \mathcal C(X)$, and $\langle X - z \rangle$ and $\langle -z+X\rangle$ are both commutative subsemigroups if $\langle X \rangle$ is.
\end{enumerate}
\end{lemma}
\begin{proof}
\ref{item:trivialieties_(i)}--\ref{item:trivialieties_(ii)} Units are cancellable elements, and for a cancellable $z \in A$ both the functions $A \to A: x \mapsto x + z$ and $A \to A: x \mapsto z + x$ are injective.

\ref{item:trivialieties_(iii)} By hypothesis, $(n+1)z = kz$ for some $k = 1, \ldots, n$. We claim that $k = 1$. Indeed, if $k \ge 2$ then $z$ being cancellable (in $\mathbb A$) implies $(n+2-k)z = z$, which is impossible, since $2 \le n+2-k \le n$ and $z, \ldots, nz$ are pairwise distinct.

So, for all $x \in A$ we have $x+z = (x + nz) + z$ and $z + (nz + x) = z + x$, which, by using again that $z$ is cancellable, yields $x + nz = nz + x = x$, and ultimately means that $\mathbb A$ is a monoid with identity $nz$.

\ref{item:trivialities_(iv)} Let $z \in \mathcal C(X)$ and $x \in X$, and for ease of notation denote by $\tilde z$ the inverse of $z$. By the cancellativity of $\mathbb A$, it is immediate that $x + \tilde z = \tilde z + x$ if and only if $x = (x + \tilde z) + z = \tilde z + x + z$, which is
true, as $\tilde z + x + z = \tilde z + z + x = x$ by the assumptions on $x$ and $z$. It follows that $\tilde z \in \mathcal C(X)$.

With this in hand, suppose $\langle X
\rangle$ is a commutative subsemigroup of $\mathbb A$ and pick $v, w \in \langle X - z
\rangle$. Then, there exist $k,\ell \in \mathbf N^+$ and $x_1, \ldots, x_k, y_1, \ldots, y_\ell \in X$ such that $v = \sum_{i=1}^k (x_i + \tilde z)$ and $w = \sum_{i=1}^\ell (y_i + \tilde z)$, with the result that $v + w = w + v$ by induction on $k + \ell$ and the observation that, for all $u_1, u_2 \in X$, it holds
$$
(u_1 + \tilde z) + (u_2 + \tilde z) = u_1 + u_2 +
2\tilde z = u_2 + u_1 + 2\tilde z = (u_2 +
\tilde z) + (u_1 + \tilde z),
$$
where we have used, in particular, that $\tilde
z \in \mathcal C(X)$, as proved in the above. Hence $\langle X -
z\rangle$ also is commutative, and the case of $\langle {-z} + X\rangle$ is analogous.
% (again, we omit further details).
\end{proof}
We omit the proof of the next elementary lemma, but the interested reader can refer to \cite[Lemma 12 and Remark 13]{Tring13} for details.
\begin{lemma}\label{lem:old_units_distribution}
If $X_1, \ldots, X_n \subseteq A$, then $X_1^\times + \cdots + X_n^\times \subseteq (X_1 + \cdots + X_n)^\times$, and the inclusion is actually an equality provided that $\mathbb A$ is cancellative.

Moreover, if $\mathbb A$ is a monoid and $x_1 \in X_1^\times, \ldots, x_n \in X_n^\times$, then $x := x_1 + \cdots + x_n$ is too an invertible element and ${-x} = (-x_n) + \cdots + (-x_1)$.
\end{lemma}
The following result reveals a certain invariance of the Cauchy-Davenport constant; we address the reader to \cite[Proposition 14]{Tring13} for a proof.
\begin{lemma}\label{lem:unital_shifts_vs_gamma_equality}
Assume $\mathbb A$ is a monoid, and let $X \subseteq A$ and $z \in A^\times$. Then
$$
\gamma(X) = \gamma(X - z) = \gamma({-z} + X).
$$
\end{lemma}
Given $X, Y \subseteq A$, we say that a pair $(X_0, Y_0)$ of subsets of $A$ is an inv\-ar\-i\-ant transform, relative to $\mathbb A$, of $(X, Y)$ if:
\begin{enumerate}[label={\rm (\textsc{s}\arabic{*})}]
\item $|X + Y| = |X_0 + Y_0|$;
\item $|X| = |X_0|$ and $|Y| = |Y_0|$;
\item $\gamma(X) = \gamma(X_0)$ and $\gamma(Y) = \gamma(Y_0)$.
\end{enumerate}
This notion is motivated by the next lemma, cf. \cite[Corollary 15]{Tring13}.
\begin{lemma}
\label{lem:units_give_invariant_n_transform}
Let $\mathbb A$ be a cancellative monoid with identity $0_\mathbb{A}$, and pick $X, Y \subseteq A$. Assume that
$Y^\times \ne \emptyset$ and $|Y| \ge 2$, and let $\kappa$ be an integer $\le \gamma(Y)$.
Then, there exists an in\-var\-i\-ant transform $(X_0, Y_0)$ of $(X, Y)$ such that:
\begin{enumerate}[label={\rm (\roman{*})}]
\item\label{lem:subinvariance(i)} $0_\mathbb{A} \in Y_0$ and $\gamma(Y_0) \ge \ord(y) \ge \kappa$ for every $y \in Y_0 \setminus \{0_\mathbb{A}\}$;
\item\label{lem:subinvariance(ii)} if $\langle Y \rangle$ is commutative, then so is $\langle Y_0 \rangle$;
\item\label{lem:subinvariance(iii)} if $\langle Y \rangle$ is commutative and $X + 2Y \neq X + Y + \bar{y}$ for every $\bar{y} \in Y^\times$, then $X_0 + 2Y_0 \ne X_0 + Y_0 + \bar{y}_0$ for all $\bar{y}_0 \in Y_0^\times$.
\end{enumerate}
\end{lemma}
\begin{proof}
Fix an integer $\kappa \le \gamma(Y)$, and using that $Y^\times \ne \emptyset$ and $|Y| \ge 2$, let $y_0 \in Y^\times$ such that
$
\gamma(Y) \ge \inf_{y_0 \fixed[0.1]{ \text{ }} \ne \fixed[0.2]{ \text{ }} y \fixed[0.15]{ \text{ }} \in Y} \ord(y - y_0) \ge \kappa$.
Then, set
$$
X_0 := X + y_0\ \ \text{and}\ \ Y_0 := -y_0 + Y.
$$
Clearly, $0_\mathbb{A}$ is in $Y_0$, and a straightforward computation gives that
\begin{equation}
\label{equ:a_simple_identity}
X + Y = (X + y_0) + ({-y_0} + Y) = X_0 + Y_0.
\end{equation}
In addition, since $y_0 \in A^\times$ and units are cancellable, we have from Lemma \ref{lem:unital_shifts_vs_gamma_equality} that $\gamma(X) = \gamma(X_0)$ and $\gamma(Y) = \gamma(Y_0)$, and from Lemma \ref{lem:trivialieties}\ref{item:trivialieties_(i)} that $|X| = |X_0|$ and $|Y| = |Y_0|$. Lastly, using that $0_\mathbb{A} \in Y_0$ and, on the other hand, $\upsilon \in Y_0$ if and only if $\upsilon = y - y_0$ for some $y \in Y$, we find
\begin{displaymath}
\gamma(Y_0) \ge \inf_{0_\mathbb{A} \fixed[0.1]{ \text{ }} \ne \fixed[0.2]{ \text{ }} \upsilon \fixed[0.15]{ \text{ }} \in Y_0} \ord(\mathfrak \upsilon) = \inf_{y_0 \fixed[0.1]{ \text{ }} \ne \fixed[0.2]{ \text{ }} y \fixed[0.15]{ \text{ }} \in Y} \ord(y - y_0) \ge \kappa.
\end{displaymath}
Putting it all together, this shows that $(X_0, Y_0)$ is an invariant transform of $(X, Y)$. So point \ref{lem:subinvariance(i)} is proved, and  \ref{lem:subinvariance(ii)} follows from Lemma \ref{lem:trivialieties}\ref{item:trivialities_(iv)}.

As for \ref{lem:subinvariance(iii)}, suppose that $\langle Y \rangle$ is commutative and $X + 2Y \ne X + Y + \bar{y}$ for all $\bar{y} \in \langle Y \rangle$, yet $X_0 + 2Y_0 = X_0 + Y_0 + \bar{\upsilon}$ for some $\bar{\upsilon} \in Y_0^\times$. Accordingly, note that $Y_0^\times = {-y_0} + Y^\times$ by Lemma \ref{lem:old_units_distribution}, and let $\bar{y} \in Y^\times$ such that $\bar{\upsilon} = {-y_0} + \bar{y}$. Then, we get from \eqref{equ:a_simple_identity} and point \ref{lem:subinvariance(ii)} above that
$$
X + Y + \bar{y} - y_0 = X_0 + 2Y_0 = X + 2Y - y_0,
$$
which yields $X + 2Y = X + Y + \bar{y}$ (again, because $-y_0$ is a unit, and hence we can cancel it out). This, however, is absurd and leads to the desired conclusion.
%since $y = \bar y + y_0 \in Y_0^\times + Y^\times \subseteq A^\times$ by Lemma \ref{lem:old_units_distribution}, and thus $y \in Y^\times$.
\end{proof}
Last but not least, we will need the following proposition, which is essentially a revised version of \cite[Proposition 23]{Tring12}.
\begin{proposition}\label{prop:properties_of_the_modified_Davenport_transform}
Assume $\mathbb A$ is a cancellative semigroup, and let $X, Y \subseteq A$ be such that $X + 2Y \not\subseteq X + Y$ and $\langle Y \rangle$ is a commutative subsemigroup of $\mathbb A$. Accordingly, fix $z \in (X + 2Y) \setminus (X + Y) \ne \emptyset$, and define
\begin{equation}\label{equ:defining_Davenport_transformed_pair}
\tilde Y_z := \{y \in Y: z \in X + Y + y\}
\ \ \text{and}\ \ %
Y_z := Y \setminus \tilde Y_z.
\end{equation}
Then the following hold:
\begin{enumerate}[label={\rm(\roman{*})}]
\item\label{item:D_transform_ii} $(X + Y_z) \cup (z - \tilde Y_z) \subseteq X + Y$;
\item\label{item:D_transform_iii} $(X + Y_z) \cap (z - \tilde Y_z) = \emptyset$;
\item\label{item:D_transform_iv} $|z - \tilde Y_z| \ge |\tilde Y_z|$;
\item\label{item:D_transform_v}  $|X + Y| + |Y_z| \ge |X + Y_z| + |Y|$.
\end{enumerate}
\end{proposition}
\begin{proof}
\ref{item:D_transform_ii} Let $w \in z - \tilde Y_z$. Then, there exists $y \in \tilde Y_z$ such that $z = w + y$. But $y \in \tilde Y_z$ if and only if $z = \tilde w + y$ for some $\tilde w \in X+Y$, so $w = \tilde w$ by
cancellativity, and hence $w \in X+Y$. This shows that $z - \tilde Y_z \subseteq X + Y$, and then we are done, as it is clear, on the other hand, that $X+Y_z \subseteq
X+Y$.
% (because $X_z \subseteq X$).

\ref{item:D_transform_iii} Suppose for a contradiction that $W := (X + Y_z) \cap (z - \tilde Y_z)$ is non-empty, and let $w \in W$. Then $w = x + y_1$ and $z = w + y_2$ for some $x \in X$, $y_1 \in Y_z$, and $y_2 \in \tilde Y_z$. Since
$\langle Y \rangle$ is commutative, it follows that
$$
z = x + y_1 + y_2 = x + y_2 + y_1,
$$
which implies by \eqref{equ:defining_Davenport_transformed_pair} that $y_1 \in \tilde Y_z$, because
$Y_z, \tilde Y_z \subseteq Y$. This is, however, absurd, as
$Y_z$ and $\tilde Y_z$ are obviously disjoint.

\ref{item:D_transform_iv} We have from
\eqref{equ:defining_Davenport_transformed_pair} that for each $y \in \tilde Y_z$
there exists $w \in X+Y$ such that $z = w + y$, and hence $w
\in z - \tilde Y_z$. On the other hand, $\mathbb A$ being cancellative yields that $w + y_1 \ne w + y_2$ for all $w \in
A$ and distinct $y_1, y_2 \in \tilde Y_z$. Thus, we see that there is an injection $\tilde Y_z \to z - \tilde Y_z$, with the result that $|z - \tilde Y_z| \ge |\tilde Y_z|$.
%(It can be shown that this inequality is indeed an equality, but that is useless here.)

\ref{item:D_transform_v} Note first that $X$ and $Y$ are non-empty, because otherwise we would have
$(X + 2Y) \setminus (X + Y) = \emptyset$, in contrast to our assumptions.

Using that $\mathbb A$ is can\-cel\-la\-tive, it follows from Lemma
\ref{lem:trivialieties}\ref{item:trivialieties_(ii)} that $|X+Y| \ge |Y|$. This implies the claim if $|Y| = \infty$, so suppose from now on that $Y$ is a finite set.

Then, the inclusion-exclusion
principle and the above points \ref{item:D_transform_ii}-\ref{item:D_transform_iv} give
\begin{equation*}
 |X+Y| \ge |X+Y_z| + |z - \tilde Y_z| \ge |X+Y_z| + |\tilde Y_z|.
\end{equation*}
But $\tilde Y_z = Y \setminus Y_z$ and $Y_z \subseteq Y$, so in the end
$
|X+Y| \ge |X+Y_z| + |Y| - |Y_z|$,
and the proof is thus complete.
\end{proof}
We conclude this section with the following:
\begin{proof}[Proof of Proposition \ref{prop:equivalent_form_with_generated_subgroup}]
\ref{cor:equ_condition_(i)} $\Rightarrow$ \ref{cor:equ_condition_(ii)}. Assume that $X + 2Y = X + Y + \bar{y}$ for some $\bar{y} \in Y$, and let $y \in Y$. Then $X + Y + y \subseteq X + Y + \bar{y}$, and on the other hand, we have from Lemma \ref{lem:trivialieties}\ref{item:trivialieties_(i)} that $|X + Y + y| = |X + Y + \bar{y}|$. But since $X$ and $Y$ are finite, this is possible only if $X + Y + y = X + Y + \bar{y}$, and we are done.

\ref{cor:equ_condition_(ii)} $\Rightarrow$ \ref{cor:equ_condition_(iii)}. Pick $\bar y \in Y^\times$. By hypothesis, we have $X + 2Y = X + Y + \bar y$, and using that $\langle Y \rangle$ is com\-mutative, this is equivalent to $X + 2(Y - \bar y) = X + Y - \bar y$.

It follows (by induction) that $X + n(Y - \bar y) = X + Y - \bar y$ for all $n \in \NNb^+$, and since $\langle W \rangle = \bigcup_{n \ge 1} nW$ for every $W \subseteq A$, we obtain
\begin{equation}
\label{equ:identity_with_subsgrp}
X + \langle Y - \bar y \rangle = X + Y - \bar y.
\end{equation}
Now, the conclusion is trivial if $X$ is empty.
%because in that case $X + \langle\fixed[-0.6]{ \text{ }}\langle Y - \bar y \rangle\fixed[-0.6]{ \text{ }}\rangle = X + \langle Y - \bar y \rangle = X + Y - \bar y = \emptyset$.
Otherwise, we get by points \ref{item:trivialieties_(i)} and \ref{item:trivialieties_(ii)} of Lemma \ref{lem:trivialieties}, equation \eqref{equ:identity_with_subsgrp}, and the assumption that $X$ and $Y$ are finite that
$$
\ord(Y - \bar y) \le |X + \langle Y - \bar y \rangle| = |X + Y - \bar y| = |X + Y| \le |X| \cdot |Y| < \infty,
$$
Thus, $\ord(y - \bar y) < \infty$ for all $y \in Y$, which, together with Lemma \ref{lem:trivialieties}\ref{item:trivialieties_(iii)}, implies $\langle Y - \bar y \rangle = \langle\fixed[-0.6]{ \text{ }}\langle Y - \bar y \rangle\fixed[-0.6]{ \text{ }}\rangle$. This leads to the desired conclusion.

\ref{cor:equ_condition_(iii)} $\Rightarrow$ \ref{cor:equ_condition_(i)}. Let $\bar y \in Y^\times$. By hypothesis, we have $X + \langle Y - \bar y \rangle = X + Y - \bar y$. Together with the commutativity of $Y$, this implies
$$
X + Y - \bar y \supseteq X + 2(Y - \bar y) = X + 2Y - 2\bar y, 
$$
and hence $X + 2Y \subseteq X + Y + \bar y$, which is enough to conclude the proof (since, of course, $X + Y + \bar y \subseteq X + 2Y$).
\end{proof}
\section{Proofs}
\label{sec:proofA062}
We start with Theorem \ref{th:mainA062}, whose proof is actually a ``transformation proof'', extending to a non-commutative setting ideas first used by H.~Davenport in \cite{Daven35}.
%his 1935 paper on the Cauchy-Davenport theorem \cite{Daven35}.

In fact, the reasoning follows the same broad scheme of the proof of \cite[Theorem 8]{Tring12}, but differs from the latter in significant details.
\begin{proof}[Proof of Theorem \ref{th:mainA062}]
Set $\kappa := |X+Y|$ for brevity's sake, and suppose that $X + 2Y \ne X + Y + \bar{y}$ for all $\bar{y} \in Y^\times$. We have to prove that
\begin{equation}
\label{equ:reduction_to_the_main_case}
\kappa \ge |X| + \min(\gamma(Y), |Y| - 1).
\end{equation}
This is obvious if $Y^\times = \emptyset$ or $|Y| = 1$, since in that case the right-hand side of \eqref{equ:reduction_to_the_main_case} equals $|X|$, and $\kappa \ge |X|$ by Lemma \ref{lem:trivialieties}\ref{item:trivialieties_(ii)}. So we assume for the sequel that $Y^\times$ is non-empty and $|Y| \ge 2$.

Then, also $X$ is non-empty, otherwise $X + 2Y = X + Y + \bar{y} = \emptyset$ for every unit $\bar{y} \in Y^\times$, in contrast to our hypotheses (as $Y^\times \ne \emptyset$). Hence, we are done if $X$ or $Y$ is infinite, since $\kappa \ge \max(|X|,|Y|)$, again by Lemma \ref{lem:trivialieties}\ref{item:trivialieties_(ii)}.

Putting it all together, we are thus reduced to the case where
\begin{equation}
\label{equ:first_summary}
1 \le |X| < \infty, \ \ 2 \le |Y| < \infty,\ \ \text{and}\ \ Y^\times \ne \emptyset,
\end{equation}
which means, among other things, that $\mathbb A$ is (necessarily) a monoid; as usual, we will denote the identity of $\mathbb A$ by $0_\mathbb{A}$.

Building on these premises, we now suppose, towards a contradiction, that
\begin{equation}
\label{equ:for_the_sake_of_contradiction}
\kappa < |X| + \min(\gamma(Y), |Y| - 1).
\end{equation}
More precisely, we assume that $(X, Y)$ is a minimal counterexample to \eqref{equ:reduction_to_the_main_case}, in the sense that if $(\bar X, \bar Y)$ is another pair of non-empty subsets of $A$ such that $\langle \bar Y \rangle$ is commutative, $|\bar Y| \ge 2$ and $\bar X + 2\bar Y \ne \bar X + \bar Y + \bar{\bar{y}}$ for every $\bar{\bar{y}} \in \bar{Y}^\times$, and
$$
|\bar X + \bar Y| < |\bar X| + \min(\gamma(\bar Y), |\bar Y| - 1),
$$
then $|Y| \le |\bar Y|$; of course, this is always possible and involves no loss of generality.
Lastly, we may further assume, as we do, that
\begin{equation}
\label{equ:special_form_of_gamma}
0_\mathbb{A} \in Y \ \ \text{and}\ \ \gamma(Y) \ge \inf_{0_\mathbb{A} \ne y \in Y} \ord(y) \ge \kappa - |X| + 1,
\end{equation}
for we get by \eqref{equ:first_summary}, \eqref{equ:for_the_sake_of_contradiction}, and Lemma \ref{lem:units_give_invariant_n_transform} that this, again, does not affect the generality of the reasoning.
Accordingly, we have that
\begin{equation*}
%\label{equ:not-included}
X + 2Y \not\subseteq X + Y.
\end{equation*}
In fact, $0_\mathbb{A} \in Y$ yields that $X + Y \subseteq X + 2Y$; therefore, $X+2Y \subseteq X+Y$ would imply $X + 2Y = X + Y + 0_\mathbb{A}$, which is, however, impossible, as we are supposing $X + 2Y \ne X + Y + \bar{y}$ for every $\bar{y} \in Y^\times$.

So, let $z$ be some element in the non-empty set $(X + 2Y) \setminus (X + Y)$, and define
$$
\tilde Y_z := \{y \in Y: z \in X + Y + y\}
\ \ \text{and}\ \ %
Y_z := Y \setminus \tilde Y_z.
$$
Clearly, $\tilde Y_z \ne \emptyset$ and $0_\mathbb{A} \notin \tilde Y_z$, so we have by \eqref{equ:special_form_of_gamma} that
$0_\mathbb{A} \in Y_z$ and $
1 \le |Y_z| < |Y|$.
Then, exploiting that $\langle Y \rangle$ is commutative and $|Y| < \infty$, we obtain by Proposition \ref{prop:properties_of_the_modified_Davenport_transform}\ref{item:D_transform_v} and equation \eqref{equ:for_the_sake_of_contradiction} that
\begin{equation}
\label{equ:transference}
|X + Y_z| \le |X+Y| + |Y_z|  - |Y| < |X| + |Y_z| - 1.
\end{equation}
It follows that $|Y_z| \ge 2$, as otherwise we would have from Lemma \ref{lem:trivialieties}\ref{item:trivialieties_(i)}
and \eqref{equ:transference} that $|X| = |X+Y_z| < |X|$, which is absurd. To summarize, we have found that
\begin{equation}
\label{equ:props_of_the_Davenport_transform}
0_\mathbb{A} \in Y_z \subsetneq Y,
\ \ 2 \le |Y_z| < |Y|,
\ \ \text{and}\ \
|X+Y_z| < |X| + |Y_z| - 1,
\end{equation}
which, along with \eqref{equ:for_the_sake_of_contradiction} and \eqref{equ:special_form_of_gamma}, gives
\begin{equation}
\label{eq:towards_a_contradiction}
|X+Y_z| \le \kappa < |X| + \inf_{0_\mathbb{A} \ne y \in Y} \ord(y) \le |X| + \inf_{0_\mathbb{A} \ne y \in Y_z} \ord(y) \le |X| + \gamma(Y_z),
\end{equation}
where we have used, in particular, that $|Y_z| \ge 2$ and $Y_z^\times \ne \emptyset$ by \eqref{equ:props_of_the_Davenport_transform}, and that $\inf(C) \le \inf(B)$ provided $\emptyset \ne B \subseteq C \subseteq \mathbf N \cup \{\infty\}$.

This is, however, absurd, as \eqref{equ:props_of_the_Davenport_transform} and \eqref{eq:towards_a_contradiction} together contradict the minimality of the pair $(X, Y)$, and it concludes the proof.
\end{proof}
Now we can proceed to prove the corollaries of Theorem \ref{th:mainA062}.
\begin{proof}[Proof of Corollary \ref{cor:UDT_improved}]
The claim is trivial if $
|X + Y| \ge |X| + \min(\gamma(Y), |Y| - 1)$, or $|Y| \le 1$, or $Y^\times$ is empty.
Otherwise, we have from Theorem \ref{th:mainA062} and Proposition \ref{prop:equivalent_form_with_generated_subgroup}\ref{cor:equ_condition_(iii)} that $
X + Y - \bar y = X + \langle\fixed[-0.6]{ \text{ }}\langle Y - \bar y \rangle\fixed[-0.6]{ \text{ }}\rangle$ for some $\bar y \in Y^\times$. Consequently, points \ref{item:trivialieties_(i)} and \ref{item:trivialieties_(ii)} of Lemma \ref{lem:trivialieties}, along with Lemma \ref{lem:old_units_distribution}, give that
\begin{equation*}
\begin{split}
|X + Y| & = |X + Y - \bar y| \ge |\langle\fixed[-0.6]{ \text{ }}\langle Y - \bar y \rangle\fixed[-0.6]{ \text{ }}\rangle| \ge \ord((y - \bar y) - (y_0 - \bar y)) = \ord(y - y_0)
\end{split}
\end{equation*}
for all $y \in Y$ and $y_0 \in Y^\times$, and this yields $|X+Y| \ge \gamma(Y)$.
\end{proof}
\begin{proof}[Proof of Corollary \ref{cor:HS}]
Of course, $X$ is non-empty, otherwise $X \cup (X + Y) = X + \langle\fixed[-0.6]{ \text{ }}\langle Y \rangle\fixed[-0.6]{ \text{ }}\rangle$. Consequently, the claim is trivial if $X$ and $Y$ is infinite, since in that case $|X \cup (X + Y)| = \infty$ by Lemma \ref{lem:trivialieties}\ref{item:trivialieties_(ii)}, and it is still trivial if $Y$ is empty, since then either side of equation \eqref{equ:hamidoune-shatrowsky_bound} is equal to $|X|$.

So, we assume for the sequel that $X$ and $Y$ are both finite and non-empty, in such a way that $X \cup (X + Y)$ is finite too, and set $Y_0 := Y \cup \{0_\mathbb{A}\}$.
 
We claim that $X + Y_0 = X \cup (X + Y) \ne X + \langle\fixed[-0.6]{ \text{ }}\langle Y_0 \rangle\fixed[-0.6]{ \text{ }}\rangle$. If $\langle\fixed[-0.6]{ \text{ }}\langle Y \rangle\fixed[-0.6]{ \text{ }}\rangle$ is infinite, this is clear from the above; otherwise,
$\langle\fixed[-0.6]{ \text{ }}\langle Y_0 \rangle\fixed[-0.6]{ \text{ }}\rangle = \langle\fixed[-0.6]{ \text{ }}\langle Y \rangle\fixed[-0.6]{ \text{ }}\rangle$ by Lemma \ref{lem:trivialieties}\ref{item:trivialieties_(iii)}, and we have (by hypothesis) $X \cup (X + Y) \ne X + \langle\fixed[-0.6]{ \text{ }}\langle Y \rangle\fixed[-0.6]{ \text{ }}\rangle$.

Therefore, we get from Theorem \ref{th:mainA062} and Proposition \ref{prop:equivalent_form_with_generated_subgroup}\ref{cor:equ_condition_(iii)} that
$$
|X \cup (X + Y)| = |X+Y_0| \ge |X| + \min(\gamma(Y_0), |Y_0| - 1),
$$
which concludes the proof, because $|Y_0| - 1 = |Y| - \mathbf 1_Y(0_\mathbb{A})$.
\end{proof}
\begin{proof}[Proof of Corollary \ref{cor:pillai&chowla_improved}]
To begin, let $\tilde{w}$ denote, for every $w \in \ZZb_n$, the smallest non-negative integer in $w$.
The claim is trivial if $Y$ is a singleton. Otherwise, since
$
\ord(w-w_0) = n/\gcd(n, \tilde{w}-\tilde{w}_0)
$
for all $w,w_0 \in \mathbf Z_n$, we have
\begin{equation*}
\begin{split}
\gamma(Y) & = \max_{y_0 \in Y} \min_{y_0 \ne y \in Y} \ord(y-y_0)  = \frac{n}{\min_{y_0 \in Y} \max_{y_0 \ne y \in Y} \gcd(n,\tilde{y}-\tilde{y}_0)}.
\end{split}
\end{equation*}
It follows that $\gamma(Y) = \delta_Y^{-1} n$, because $\gcd(n, \tilde y) = \gcd(n, \xi)$ for every $y \in \mathbf Z_n$ and $\xi \in \ZZb$ with $\xi \equiv \tilde y \bmod n$. So we are done by Theorem \ref{th:mainA062} and Proposition \ref{prop:equivalent_form_with_generated_subgroup}\ref{cor:equ_condition_(ii)}.
\end{proof}
\section{Closing remarks}
\label{sec:remarks}
The bound provided by Theorem \ref{th:mainA062}\ref{cor:main_condition_(i)} is meaningful only if $\gamma(X) > 0$,
insofar as $\mathbb A$ being a cancellative semigroup implies, by Lemma \ref{lem:trivialieties}\ref{item:trivialieties_(ii)}, that $|X + Y| \ge |X|$. This means, in particular, that the theorem is not very useful unless $\mathbb A$ is a monoid, and raises the challenge of further generalizing the result (and its corollaries) so as to replace $X^\times$ in \eqref{equ:CD_constant} with a subset of $A$ that is significant also when $A^\times = \emptyset$.

On a similar note, every \textit{commutative} cancellative semigroup can be embedded into a group.
It was, however, proved by A.~Malcev in \cite{Malcev37} that there are finitely generated cancellative semigroups that do \textit{not} embed into a group, which serves as a ``precondition'' for some aspects of the present work and its prequels \cite{Tring12, Tring13}, as
it shows that the study of sumsets in cancellative semigroups cannot be systematically
reduced,
in the absence of commutativity, to the case of groups (at least, not in an
obvious way).
\section{Acknowledgments}
This research was supported by the Austrian FWF Project M1900-N39, and partly by the French ANR Project ANR-12-BS01-0011. The author is grateful to Paolo Leonetti (Universit\`a Bocconi, Italy) for some useful comments.

\end{document}